\documentclass[a4paper,11pt]{amsart}
\usepackage{amsmath}
\usepackage{amsthm}
\usepackage{amssymb}
\usepackage[latin1]{inputenc}
\usepackage[all,arc]{xy}
\usepackage[english]{babel}
\usepackage{color}

\theoremstyle{plain}
\newtheorem{thm}{Theorem}[section]
\newtheorem{prop}[thm]{Proposition}
\newtheorem{lem}[thm]{Lemma}
\newtheorem{cor}[thm]{Corollary}

\newtheorem*{assumption}{Assumption}
\theoremstyle{definition}

\theoremstyle{remark}
\newtheorem{ex}[thm]{Example}
\newtheorem{rem}[thm]{Remark}

\DeclareMathOperator{\scal}{scal}
\DeclareMathOperator{\Ric}{Ric}

\begin{document}
\title{Geometric flows and K\"ahler reduction}
\date{\today}

\author{Claudio Arezzo}
\address{The Abdus Salam International Centre for Theoretical Physics\\ Trieste, Italy
and \newline Universit\`a degli Studi di Parma, Italy.}
\email{arezzo@ictp.it}
\author{Alberto Della Vedova} 
\address{Universit\`a degli Studi di Parma, Italy.}
\email{alberto.dellavedova@unipr.it}
\author{Gabriele La Nave}
\address{University of Illinois, Urbana-Champaign}
\email{lanave@illinois.edu}

\begin{abstract}
We investigate how to obtain various flows of K\"ahler metrics on a fixed manifold as variations of K\"ahler reductions of a metric satisfying a given static equation on a higher dimensional manifold. We identify static equations that induce the geodesic equation for the Mabuchi's metric, the Calabi flow, the pseudo-Calabi flow of Chen-Zheng and the K\"ahler-Ricci flow. In the latter case we re-derive the V-soliton equation of La Nave-Tian. 
\end{abstract}

\maketitle

\section{Introduction}

This note is concerned with the description, in various instances, of what geometry one should impose on the total space of a K\"ahler manifold $(P,\omega, J)$ endowed with a Hamiltonian holomorphic circle action with moment map $\mu :P \to \mathbb R$, so that the {\it variation} of symplectic quotients with induced metrics, sometimes called \emph{K\"ahler reductions}, describes indeed a given flow on the initial quotient (naturally, up to diffeoforphisms). The mathematical literature is by now rich of examples of the interplay between geometry and K\"ahler (or merely symplectic) reductions, spanning from Guillemin and Sternberg fundamental papers \cite{gs1, gs2, gs3} to the more recent work of Burns and Guillemin \cite{bg} passing through the work of Futaki \cite{Fut}.

Our motivation for this study is twofold: one is based on trying to create a theoretical set-up and machinery where one can naturally incorporate flows with {\it surgery} in a context where such surgeries occur naturally in the Morse-theory of variations of K\"ahler reductions; the other is to try and understand convergence (or lack thereof) at infinite time in terms of what happens to the K\"ahler reduction when one meets the first critical point of the moment map (which is not necessarily the first time the reduction becomes "singular", cf. Example \ref{sing-quot}). 

To the authors knowledge, the first instance in which K\"ahler reduction was used to analyze the nature of finite time singularities of a given flow of metrics is in the work of La Nave and Tian, in connection to the K\"ahler-Ricci flow \cite{LT}. There it is shown that the K\"ahler-Ricci flow on a K\"ahler manifold $M$ is loosely speaking equivalent to a static equation, dubbed the $V$-soliton equation since it is of soliton type, on a manifold $P$ endowed with a Hamiltonian holomorphic  circle action, which has $M$ as one of its K\"ahler reductions.

We continue with this idea in this paper, where we investigate in fact how to obtain various flows of metrics as variations of K\"ahler reductions of a metric on the total space satisfying a given static equation.
This approach is somewhat different from the usual one inspired by Kaluza-Klein theories, where one usually is interested in the rich nature of quotients of a given special metric, such as a K\"ahler-Einstein metric (cf. \cite{bg,Fut}). Our take could be characterized as being somewhat opposite and more along the lines of a Geometric-flow version of variation of GIT or symplectic quotients (cf. \cite{gs3} and references therein) whereby we analyze the {\it variation} of metrics under the variation of reductions and identify equations on the total space for which the variation of metrics follows a specified flow. 

Specifically, we identify equations on the total space that induce:
\begin{itemize}
\item
 the Geodesic equation in the space of K\"ahler potentials (with respect to Mabuchi $L^2$-metric) in Theorem \ref{thm::Geodesics}; 
 \item
 the Calabi flow in Theorem \ref{thm::Calabiflow}; 
 \item
  the pseudo-Calabi flow of Chen and Zheng in Theorem \ref{thm::pseudoCalabiflow};
 \item
 and finally we re-derive La Nave and Tian's $V$-soliton equation (cf. \cite{LT}) in Theorem \ref{thm::nKRflow} for the normalized
 K\"ahler-Ricci flow and its variant for the unnormalized one in Theorem \ref{thm::KRflow}.
\end{itemize}

Arguably, the major difference in approach between the current paper and \cite{LT} is the fact that for most of our applications, such as the Calabi flow, one needs only --and in fact must--  take the space whose K\"ahler reductions represent the flow to simply be a product $P= M\times A$ with  $A\subset \mathbf C^*$ some annulus (cf. Theorem \ref{theorem::total-space} in section \ref{section::total-space}). In the case of the normalized K\"abler-Ricci flow this is also true, as long as the initial metric is canonical or anticanonical, as the case maybe. In particular, the case of interest in \cite{LT}, namely the finite-time singularity of the K\"ahler-Ricci flow in non-(anti)-canonically polarized metrics, is never of this nature.

In the last section of this paper we explain how these new equations are indeed equivalent to the associated geometric flows; namely, given a solution to the flow on $M$, how to construct the relevant structures on the product $P$ to get a solution to the new equations.

Along the way, we show how natural geometric quantities of the reduced metric $g_\tau$ on $M$ comes from \emph{reduction} (see section \ref{ss::Kr} for precise definition) of suitable quantities on $P$.
For example in Proposition \ref{prop::riccired} and Corollary \ref{cor::scalred} we show that there exist a one form $\rho$ and a smooth function $R$ on $P$ that descend respectively to the Ricci form and to the scalar curvature of the reduced metric $g_\tau$ on $M$.
In the same circle of ideas, we show in Lemma \ref{lem::MAred} and Corollary \ref{cor::Deltared} that the Monge-Amp\`ere operator and, in particular, the Laplacian of the reduced metric $g_\tau$ are induced by suitable operators on $P$.

\section{Hamiltonian circle actions on K\"ahler manifolds}\label{sec::Hamcircact}

We are mainly interested in Hamiltonian holomorphic circle actions on K\"ahler manifolds.
The purpose of this section is to investigate the geometry of the K\"ahler reduction in connection with the one of the starting manifold. 
For simplicity we consider just the case of a semi-free action. This means that it is free away from the fixed point set, or equivalently that there are no finite non-trivial isotropy groups.
In this case, any non-fixed point has a neighborhood equivariantly biholomorphic to the product of a fixed manifold with an annulus endowed with the standard circle action.
The reduced manifold can be recovered by patching together the reductions of these invariant neighborhoods.
Since we are just interested in local properties of the reduced manifolds, we can assume that our starting manifold is itself a product.
In particular, this argument assures that all the statement of this sections are true for any K\"ahler manifold endowed with a semi-free holomorphic Hamiltonian circle action, with the exception of Lemma \ref{lem::HamAct} and Proposition \ref{prop::redform} (which will only hold locally).

\bigskip 
\subsection{The total space is a product}
Given a (non-necessarily compact) connected complex $n$-fold $M$, consider the product $P = M \times A$, where $A = \{ w \in \mathbf C \mbox{ s.t. } r < |w| < R\}$ is the annulus of radii $r,R > 0$.
We allow $r=0$ or $R=+\infty$.
The standard circle action on $P$ defined by $e^{i\theta} \cdot (x,w) = (x,e^{i \theta} w)$ is generated by the real vector field
$$ V = i \left( w \frac{\partial}{\partial w} - \bar w \frac{\partial}{\partial \bar w} \right). $$
Moreover one has
\begin{equation}\label{eq::JV}
JV = - \left( w \frac{\partial}{\partial w} + \bar w \frac{\partial}{\partial \bar w} \right) = - 2s \frac{\partial}{\partial s},
\end{equation}
where $J$ denotes the complex structure of $P$ and $s$ is the smooth function on $P$ defined by $s(x,w) = w \bar w$.
We will consider K\"ahler metrics on $P$ that makes Hamiltonian the standard circle action.
By definition the K\"ahler form $\omega$ of such a metric satisfies $i_V \omega = d\mu$ for some moment map $\mu : P \to \mathbf R$.
Let $\pi : P \to M$ be the projection on the first factor.

\begin{lem}\label{lem::HamAct}
Let $g$ be a K\"ahler metric on $P$ with K\"ahler form $\omega$. The standard circle action is Hamiltonian with moment map $\mu : P \to \mathbf R$ if and only if there exist a K\"ahler form $\sigma$ on $M$, a smooth invariant function $\phi : P \to \mathbf R$, and a constant $c \in \mathbf R$ such that
\begin{enumerate}\renewcommand{\theenumi}{\roman{enumi}}
\item $ \omega = \pi^*\sigma + dd^c \phi $, \label{omega}
\item $\mu = JV (\phi) + c$. \label{mu}
\end{enumerate}
Any other such triple $(\tilde \sigma, \tilde \phi, \tilde c)$ satisfies
\begin{equation}\label{eq::tildata}
\left\{
\begin{array}{l}
\tilde \sigma = \sigma + dd^c u \\
\tilde \phi = \phi - \pi^*u + \frac{\tilde c - c}{2} \log s + b,
\end{array}
\right.
\end{equation}
for some $b \in \mathbf R$, and some smooth function $u$ on $M$ such that $\sigma + dd^c u > 0$.
\end{lem}
\begin{proof}
Assuming there exists $(\sigma, \phi, c)$ as in the statement such that \eqref{omega} and \eqref{mu} hold, one has
$$ i_V \omega = L_V d^c \phi - d i_V d^c \phi = d JV(\phi) = d\mu. $$
Since $\omega$ is closed, the standard circle action turns out to be Hamiltonian with moment map $\mu$.

For any $(\tilde \sigma, \tilde \phi, \tilde c)$ giving the same $\omega$ and $\mu$, one has
\begin{equation}\label{sistema} 
\left\{
\begin{array}{l}
\pi^*(\tilde \sigma - \sigma) + dd^c(\tilde \phi - \phi) = 0 \\
JV(\tilde \phi - \phi) + \tilde c - c = 0
\end{array}
\right.
\end{equation}
Thanks to \eqref{eq::JV}, from second equation one easily finds that
$$ \tilde \phi - \phi = \frac{\tilde c - c}{2} \log s + b -\pi^*u, $$
for some constant $b \in \mathbf R$ and some smooth function $u$ on $M$.
Substituting in the first equation of \eqref{sistema} then proves \eqref{eq::tildata}.

It remains to show that any K\"ahler form $\omega$ on $P$ making Hamiltonian the standard circle action with moment map $\mu$ satisfies \eqref{omega} and \eqref{mu} for some triple $(\sigma,\phi,c)$.  
To this end consider the decomposition
\begin{equation}\label{eq::omegadec}
\omega = \eta + \alpha \wedge d^c \log s + d \log s \wedge \beta + \gamma d\log s \wedge d^c \log s, 
\end{equation}
where $\eta$ is a smooth section of the bundle $\pi^*\Lambda^{1,1}T^*M$, $\alpha,\beta$ are smooth sections of $\pi^*T^*M$, and $\gamma$ is a smooth function on $P$.
Imposing that $\mu$ is a moment map gives
$$ d\mu = i_V \omega = -2\alpha -2\gamma d\log s, $$
$$ d^c \mu = i_{JV} \omega = -2\beta -2 \gamma d^c \log s, $$
$$ JV(\mu) = \omega (V,JV) = 4\gamma, $$
whence, substituting in \eqref{eq::omegadec} after easy calculations it follows
\begin{equation}\label{eq::omegadec2}
\omega = \eta - \frac{1}{2} d\mu \wedge d^c \log s -\frac{1}{2} d \log s \wedge d^c \mu - \frac{1}{4} JV(\mu) d\log s \wedge d^c \log s.
\end{equation}
The smooth circle invariant function $\phi$ on $P$ defined by
$$ \phi(x,s) = \int_{r^2}^{s} \frac{\mu(x,t)-c}{2t}dt, $$
clearly satisfies \eqref{mu}.
On the other hand, one can easily check that the following decomposition holds (cf. \cite[Lemma 3.4]{LT}):
\begin{equation}\label{eq::ddcphi}
dd^c \phi = d_Md^c_M \phi -\frac{1}{2} d \mu \wedge d^c \log s -\frac{1}{2} d \log s \wedge d^c \mu -\frac{1}{4} JV(\mu) d \log s \wedge d^c \log s,
\end{equation}
where $d_Md^c_M \phi$ denotes the projection of $dd^c \phi$ on the sub-bundle $\pi^*\Lambda^{1,1}T^*M \subset \Lambda^{1,1}T^*P$.
Comparing $dd^c \phi$ with \eqref{eq::omegadec2} gives
\begin{equation}
\omega - dd^c \phi = \eta - d_Md^c_M \phi.
\end{equation}
The form $\omega - dd^c \phi$ is clearly circle invariant, and from the right hand side of equation above it follows that it vanish on the distribution generated by $V$ and $JV$. Moreover one has
\begin{multline*}
L_{JV}(\omega - dd^c \phi)
= d(i_{JV}\omega - i_{JV} dd^c \phi) \\
= dd^c \mu - d i_{JV} \left( d_Md^c_M \phi - \frac{1}{2} d \mu \wedge d^c \log s - \frac{1}{2} d \log s \wedge d^c \mu - \frac{1}{4} JV(\mu) d \log s \wedge d^c \log s \right)=0.
\end{multline*}
Thus there exists a K\"ahler form $\sigma$ on $M$ such that $\omega - dd^c \phi = \pi^* \sigma$, whence \eqref{omega} follows.
\end{proof}

\subsection{K\"ahler reduction}\label{ss::Kr}
Keeping notation of section 	\ref{sec::Hamcircact}, fix a K\"ahler metric $g$ on $P$ that makes Hamiltonian the standard circle action with moment map $\mu$.
Thanks to Lemma \ref{lem::HamAct}, this amounts to choosing a K\"ahler form $\sigma$ on $M$, a smooth invariant function $\phi$ on $P$, and a constant $c \in \mathbf R$ satisfying conditions \eqref{omega} and \eqref{mu} of Lemma \ref{lem::HamAct}.
Let $\tau$ be a regular value of $\mu$, so that the circle action is free on the level set $S_\tau=\mu^{-1}(\tau)$.
With no additional assumption on $\mu$, the level set $S_\tau$ is not necessarily compact even if $M$ is, and its quotient 
$$ M_\tau = \pi (S_\tau) \subset M $$
may not be closed.
Such a situation is illustrated by the following example.
\begin{ex}\label{sing-quot}
Consider the family $\pi : X \to \mathbf C$, where 
$$ X= \left\{ (z,w) \in \mathbf {CP}^2 \times \mathbf C \mbox{ s.t. } z_0z_2 - z_1^2 w = 0 \right\},$$
and $\pi$ is just the coordinate $w$.
Is easy to check that endowing $\mathbf{CP}^2$ with the circle action $e^{i\theta} \cdot z = (e^{i\theta}z_0,z_1,z_2)$, and $\mathbf C$ with the standard one, makes $X$ invariant for the diagonal action on $\mathbf {CP}^2 \times \mathbf C$. Thus $X$ has a circle action and one readily verifies that $\pi$ turns out to be an equivariant map.
Moreover it is easy to see that removing the central fiber $\pi^{-1}(0)$ from $X$ gives rise to an invariant submanifold $X^*$ biholomorphic to $P=\mathbf {CP}^1 \times \mathbf C^*$.
An explicit bihomlomorphism $\Phi: P \to X^*$ is given by
$$ \Phi(x,w) = (wx_0^2,x_0x_1,x_1^2,w). $$
Note that $\Phi$ becomes equivariant when $P$ is endowed with the standard action on the $\mathbf C^*$ factor, thus we are in the general situation considered above.

The product metric $\omega_{FS} + \omega_{Euclid}$ on $\mathbf {CP}^2 \times \mathbf C$ is Hamiltonian with moment map
$$ (z,w) \mapsto -\frac{|z_0|^2}{|z|^2} - |w|^2, $$
whence it follows readily that a moment map for the circle action on $P$ (endowed with the induced metric) is given by
$$ \mu(x,w) = -|w|^2 \frac{(1+|w|^2)|x_0|^4 + |x_0|^2|x_1|^2 + |x_1|^4}{|w|^2|x_0|^4+|x_0|^2|x_1|^2+|x_1|^4}. $$  
By means of easy calculations one can check that the image of $\mu$ is the interval $(-\infty,0)$, and one has biholomorphisms $M_\tau \simeq \mathbf {CP}^1$ for $\tau < -1$, and $M_\tau \simeq \mathbf {CP}^1 \setminus \{(0:1)\}$ for $-1 \leq \tau < 0$.
Moreover the reduced metric $\omega_\tau$ on $M_\tau$ is smooth for $\tau <-1$, conic of angle $\pi$ for $\tau =-1$ and it is incomplete for $\tau \in (-1,0)$.
In the latter case, its completion recovers $\mathbf {CP}^1$.

Finally note that $\mu$ has no critical points, since the circle action on $P$ has no fixed points. Nonetheless $\mu$ is clearly not {\it proper}.
\end{ex}

The example above shows that, due to the non-properness of the moment map $\mu$, the topology of reductions can still change when $\tau$ varies.
On the other hand, if the reduced manifold $M_\tau$ coincides with $M$, then the same holds for $M_{\tau + \varepsilon}$ for all $\varepsilon$ sufficiently small.

This example suggests to define, when $\mu$ is not proper,  the K\"ahler reduction $(M_\tau , \omega _\tau)$ as the metric completion of the standard reduction. On the other hand, for the purposes of this paper, we only need the following:
\begin{assumption}
$\tau \in \mu(P)$ is chosen so that $M_\tau$ is coincides with $M$. When $\tau$ varies, it does so in an interval where all the reductions coincide with $M$.  
\end{assumption}

Let $\iota_\tau : S_\tau \to P$ be the inclusion and let $\pi_\tau : S_\tau \to M_\tau$ be the projection on the quotient, namely the restriction of $\pi$ to the level set $S_\tau$, so that one has the following commutative diagram:
$$
\xymatrix{ 
S_\tau \ar[d]^{\pi_\tau} \ar[r]^{\iota_\tau} & P \ar[d]^\pi \\
M_\tau \ar[r] & M
}
$$

Any invariant smooth function $f$ on $P$ descends to a smooth function $f_\tau$ on $M_\tau$ defined by $$\iota_\tau^* f = \pi_\tau^*f_\tau.$$
A simple but crucial example is constituted by the function $s$, which descends to a function $s_\tau$.
Given any $f$, one can express $f_\tau$ by means of $s_\tau$.
Indeed, thanks to invariance, $f$ has the form $f=f(\pi,s)$, whence obviously it follows $f_\tau(x) = f(x,s_\tau(x))$. 

More generally, any invariant differential $r$-form $\eta$ on $P$ satisfying $\iota_\tau^*(i_V \eta) = 0,$ descends to an $r$-form $\eta_\tau$ on $M_\tau$ defined by
$$ \iota_\tau^*\eta = \pi_\tau^*\eta_\tau. $$
This defines a reduction map, that is a linear map
$$ \left\{ \eta \in \Omega^*(P) \mbox{ s.t. } L_V \eta = 0, \,\, \iota_\tau^*(i_V \eta) = 0 \right\} \to \Omega^*(M_\tau), 
\quad \eta \mapsto \eta_\tau,$$
satisfying
\begin{equation}\label{eq::homDGA}
(\eta \wedge \xi)_\tau = \eta_\tau \wedge \xi_\tau,
\qquad
(d\eta)_\tau = d\eta_\tau,
\qquad
(\pi^*\zeta)_\tau = \zeta,
\end{equation}
as can be readily checked.

In particular, the K\"ahler form $\omega$ of the metric $g$ descends to a non-degenerate two form $\omega_\tau$ on $M_\tau$.
Moreover $\omega_\tau$ turns out to be compatible with the complex structure of $M_\tau$ induced by the one of $M$, so that it defines a K\"ahler metric $g_\tau$ on $M_\tau$.

\begin{lem}\label{lem::dertau}
For every invariant functions $f \in C^\infty(P)$ one has:
\begin{equation}
\frac{\partial f_\tau}{\partial \tau}= \left(\frac{JV(f)}{|V|^2}\right)_\tau.
\end{equation}
\end{lem}
\begin{proof}
Differentiating the identity $f_\tau(x) = f(x,s_\tau(x))$ gives
\begin{equation}\label{derftauparial}
\frac{\partial f_\tau}{\partial \tau} = \left(\frac{\partial f}{\partial s}\right)_\tau \frac{\partial s_\tau}{\partial \tau} = -\frac{1}{2} \left(JV(f)\right)_\tau \frac{\partial \log s_\tau}{\partial \tau}.
\end{equation}
Taking $f=\mu$, since $\mu_\tau=\tau$ and $JV(\mu)=|V|^2$, one gets
\begin{equation}\label{eq::dertaustau}
\frac{\partial \log s_\tau}{\partial \tau} = \left( - \frac{2}{|V|^2} \right)_\tau.
\end{equation}
Thus the thesis follows substituting \eqref{eq::dertaustau} in \eqref{derftauparial}.
\end{proof}

\begin{lem}\label{lem::dcFtau}
For every invariant function $f \in C^\infty(P)$, the following holds
\begin{equation}
d^c f_\tau = \left( d^c f - \frac{JV(f)}{|V|^ 2} d^c \mu \right)_\tau.
\end{equation}
\end{lem}
\begin{proof}
Differentiating the identity $f_\tau(x) = f(x,s_\tau(x))$ gives
\begin{equation*}
d^c f_\tau 
= \left( d^c f - \frac{\partial f}{\partial s} d^c s\right)_\tau + 
\left(\frac{\partial f}{\partial s}\right)_\tau d^c s_\tau
= \left( d^c f +\frac{1}{2} JV(f) d^c \log s\right)_\tau - \frac{1}{2} \left(JV(f)\right)_\tau d^c \log s_\tau.
\end{equation*}
Taking $f=\mu$, since $\mu_\tau=\tau$ and $JV(\mu)=|V|^2$, one gets
\begin{equation}\label{eq::dcstau}
d^c \log s_\tau = \left( \frac{2}{|V|^2} d^c \mu + d^c \log s\right)_\tau,
\end{equation}
whence the thesis follows substituting in equation above.
\end{proof}

\begin{prop}\label{prop::redform}
In the situation of Lemma \ref{lem::HamAct}, the K\"ahler form of the reduced metric $g_\tau$ satisfies
$$ \omega_\tau = \sigma + dd^c \psi_\tau, $$
where 
$$ \psi = \phi + \frac{\mu-c}{2} \log s.$$
\end{prop}
\begin{proof}
By Lemma \ref{lem::dcFtau}, recalling that $JV(\phi) = \mu-c$ and $JV(\mu)=|V|^2$ it follows 
$$ d^c \psi_\tau = \left( d^c\phi + \frac{\mu-c}{2}d^c\log s \right)_\tau.$$
Thus, since $\omega = \pi^*\sigma + dd^c \phi$, by \eqref{eq::homDGA} one has
\begin{equation}
\sigma + dd^c \psi_\tau = \left( \pi^*\sigma + dd^c\phi + \frac{1}{2} d\mu \wedge d^c\log s \right)_\tau = \omega_\tau,
\end{equation}
where we used $(d\mu \wedge d^c\log s)_\tau = 0$.
\end{proof}

The Monge-Amp\`ere operator of the reduced manifold is given by the reduction of a non-linear operator on $P$. Indeed one has the following (which is a consequence of Lemma 3.3 and Lemma 3.7 in \cite{LT}, but here we produce a different proof):

\begin{lem}\label{lem::MAred}
For all invariant functions $f \in C^\infty(P)$ one has:
\begin{equation}\label{eq::MAred}
\frac{(\omega_\tau + d d^c f_\tau)^n}{\omega_\tau^n} = \left( \frac{\left(\omega + d d^c f - d\frac{JV(f)}{|V|^2}\wedge d^c\mu - \frac{JV(f)}{|V|^2}dd^c\mu \right)^{n+1}}{\omega^{n+1}} \right)_\tau.
\end{equation}
\end{lem}
\begin{proof}
Any top form $\theta$ on $P$ can be written as
\begin{equation}\label{eq::topformdec}
\theta = \eta \wedge d \log s \wedge d^c \log s,
\end{equation}
where $\eta$ is a section of $\pi^* K_M$.
Clearly one can recover $\eta$ from $\theta$ by contraction with $V$ and $JV$:
\begin{equation}\label{eq::etabycontraction} i_{JV}i_V \theta 
= -2 i_{JV} \left(\eta \wedge d \log s\right)
= 4 \eta.
\end{equation}
Let 
$$ \xi = \omega + d d^c f - d\frac{JV(f)}{|V|^2}\wedge d^c\mu - \frac{JV(f)}{|V|^2}dd^c\mu. $$
Note that it can be also written as
$$ \xi = \omega + d \left( d^c f - \frac{JV(f)}{|V|^2} d^c \mu \right). $$
We want to write $\xi^{n+1}$ in the form \eqref{eq::topformdec}.
Since 
\begin{equation}\label{eq::contractionxin+1}
i_{JV} i_V \xi^{n+1} = (n+1) i_{JV} \left( i_V \xi \wedge \xi^n \right) = (n+1) i_{JV} i_V \xi \left( \xi - \frac{i_V \xi \wedge i_{JV} \xi}{i_{JV} i_V \xi} \right)^n,
\end{equation}
we need to calculate the contraction of $\xi$ with $V$ and $JV$.
The one form $d^c f - \frac{JV(f)}{|V|^2} d^c \mu$ is circle invariant, thus
$$ i_V \xi = d\mu + L_V \left( d^c f - \frac{JV(f)}{|V|^2} d^c \mu \right) + d i_V \left( d^c f - \frac{JV(f)}{|V|^2} d^c \mu \right) = d\mu, $$
whence it follows readily that $i_{JV} i_V \xi = |V|^2$.
As we will see, we do not need an explicit expression for $i_{JV}\xi$.
Substituting in \eqref{eq::contractionxin+1} gives
\begin{equation}
i_{JV} i_V \xi^{n+1} = (n+1)|V|^2 \left( \omega + d d^c f - d\frac{JV(f)}{|V|^2}\wedge d^c\mu - \frac{JV(f)}{|V|^2}dd^c\mu - \frac{d\mu \wedge i_{JV}\xi}{|V|^2} \right)^n,
\end{equation}
and in particular, taking $f=0$:
$$ i_{JV} i_V \omega^{n+1} = (n+1) |V|^2\left(\omega - \frac{d\mu \wedge d^c \mu}{|V|^2} \right)^n, $$
by \eqref{eq::topformdec} and \eqref{eq::etabycontraction} it follows
\begin{equation}
\frac{\xi^{n+1}}{\omega^{n+1}} = \frac{\left( \omega + d d^c f - d\frac{JV(f)}{|V|^2}\wedge d^c\mu - \frac{JV(f)}{|V|^2}dd^c\mu - \frac{d\mu \wedge i_{JV}\xi}{|V|^2} \right)^n}{\left(\omega - \frac{d\mu \wedge d^c \mu}{|V|^2} \right)^n},
\end{equation}
whence the thesis follows by reducing at $\tau$ and using Lemma \ref{lem::dcFtau}.
\end{proof}
An immediate consequence is the following fact (cf. formula (14) in \cite{LT}):
\begin{cor}\label{cor::Deltared}
For all invariant functions $f \in C^\infty(P)$ the following holds:
\begin{equation}
\Delta_\tau f_\tau = \left( \Delta f - \left( \Delta \mu - JV \log |V| \right)\frac{JV(f)}{|V|^2} - \frac{(JV)^2(f)}{2|V|^2} \right)_\tau,
\end{equation}
where $\Delta$ and $\Delta_\tau$ denote the Laplace operators of the metrics $g$ and $g_\tau$ respectively.
\end{cor}
\begin{proof}
As well known the linearization of the Monge-Amp\`ere operator, i.e. the left hand side of \eqref{eq::MAred}, is twice the Laplacian. Thus by Lemma \ref{lem::MAred} one has
\begin{equation}\label{eq::Deltapartial}
2 \Delta_\tau f_\tau = \left( \frac{(n+1)\left(d d^c f - d\frac{JV(f)}{|V|^2}\wedge d^c\mu - \frac{JV(f)}{|V|^2}dd^c\mu \right) \wedge \omega^n}{\omega^{n+1}} \right)_\tau.
\end{equation}
Thanks to the identity $\nabla f \cdot \nabla g \, \omega^{n+1} = (n+1) df \wedge d^c g \wedge \omega^n$ and recalling that $\nabla \mu = JV$ one calculates
\begin{multline*}
(n+1)\left(d d^c f - d\frac{JV(f)}{|V|^2}\wedge d^c\mu - \frac{JV(f)}{|V|^2}dd^c\mu \right)\wedge \omega^n \\
= \left( 2\Delta f - JV \left(\frac{JV(f)}{|V|^2}\right) - 2 \frac{JV(f)}{|V|^2} \Delta \mu \right) \omega^{n+1} \\
= \left( 2\Delta f - \frac{JV^2(f)}{|V|^2} + \frac{JV(f)}{|V|^2} JV \log |V|^2 - 2 \frac{JV(f)}{|V|^2} \Delta \mu \right) \omega^{n+1},
\end{multline*}
whence the thesis follows by substitution in \eqref{eq::Deltapartial}.
\end{proof}

The Ricci curvature of the reduced metric turns out to be the reduction of a two form on $M$. More precisely one has the following

\begin{prop}\label{prop::riccired}
The (1,1)-form on $P$ 
$$ \rho = \Ric(\omega) + dd^c \log |V| + d\left( \frac{\Delta \mu - JV \log |V|}{ |V|^2} d^c \mu \right) $$
satisfies $\rho_\tau = \Ric(\omega_\tau)$.
\end{prop}
\begin{proof}
In the situation of Lemma \ref{lem::HamAct}, on $M$ one has
\begin{equation}
2 \Ric(\omega_\tau) - 2 \Ric(\sigma) = - dd^c \log \frac{\omega_\tau^n}{\sigma^n}.
\end{equation}
Letting $\psi = \phi - \frac{\mu - c}{2} \log s$, thanks to Proposition \ref{prop::redform} $\omega_\tau = \sigma + dd^c \psi_\tau$ so that
\begin{equation} 2 \Ric(\omega_\tau) - 2 \Ric(\sigma) = dd^c \log \frac{\omega_\tau-dd^c \psi_\tau}{\omega_\tau^n},
\end{equation}
whence by Lemma \ref{lem::MAred} it follows
\begin{equation}\label{diffricci} 2 \Ric(\omega_\tau) - 2 \Ric(\sigma) = dd^c \log \left( \frac{\left(\omega - d d^c \psi + d\frac{JV(\psi)}{|V|^2}\wedge d^c\mu + \frac{JV(\psi)}{|V|^2}dd^c\mu \right)^{n+1}}{\omega^{n+1}} \right)_\tau.
\end{equation}
As one can easily verify $JV(\psi) = \frac{|V|^2}{2}\log s$, and 
$$ dd^c \psi = dd^c \phi + \frac{1}{2} \log s dd^c \mu + \frac{1}{2} d\log s \wedge d^c \mu + \frac{1}{2} d\mu \wedge d^c \log s. $$
Thus by substitution in \eqref{diffricci}, after some easy calculations and recalling that $\omega = \pi^* \sigma + dd^c \phi$, one gets
\begin{equation} 2 \Ric(\omega_\tau) - 2 \Ric(\sigma) = dd^c \log \left( \frac{\left(\pi^* \sigma - \frac{1}{2} d\mu \wedge d^c \log s \right)^{n+1}}{\omega^{n+1}} \right)_\tau.
\end{equation}
Let $F=\left(\pi^* \sigma - \frac{1}{2} d\mu \wedge d^c \log s \right)^{n+1} / \omega^{n+1}$, by Lemma \ref{lem::dcFtau} and recalling that the exterior derivative commute with the reduction map, one has
\begin{equation}\label{diffricci2}
2 \Ric(\omega_\tau) - 2 \Ric(\sigma) = \left( dd^c \log F - d\left(\frac{JV \log F}{|V|^2} d^c \mu\right)\right)_\tau.
\end{equation}
Since $JV(\mu) = |V|^2$, it holds
$$ \left(\pi^* \sigma - \frac{1}{2} d\mu \wedge d^c \log s \right)^{n+1} = \frac{n+1}{4} |V|^2 \pi^* \sigma^n  \wedge d \log s \wedge d^c \log s, $$
whence
\begin{equation}
\log F 
= \log \frac{n+1}{4} + \log |V|^2 - \log \frac{\omega^{n+1}}{\pi^* \sigma^n  \wedge d \log s \wedge d^c \log s}.
\end{equation}
The top form $\pi^* \sigma^n  \wedge d \log s \wedge d^c \log s$ is, up to a positive constant factor, the volume form of the product metric defined by $\pi^*\sigma + d \log s \wedge d^c \log s$ on $P$.
Since $\Ric(\pi^*\sigma + d \log s \wedge d^c \log s) = \pi^* \Ric(\sigma)$, it turns out
\begin{equation}\label{ddclogF}
dd^c \log F = dd^c \log |V|^2 + 2 \Ric(\omega) - 2 \pi^* \Ric(\sigma).
\end{equation}
On the other hand, since $L_{JV}(\pi^* \sigma^n  \wedge d \log s \wedge d^c \log s) = 0$, one calculates
\begin{equation}\label{JVlogF}
JV \log F 
= JV \log |V|^2 - \frac{L_{JV}(\omega^{n+1})}{\omega^{n+1}}
= JV \log |V|^2 - 2\Delta \mu. 
\end{equation}
The thesis follows substituting \eqref{ddclogF}, \eqref{JVlogF} in \eqref{diffricci2} and observing that $(\pi^* \Ric(\sigma))_\tau = \Ric(\sigma)$. 
\end{proof}

The scalar curvature 
$$
\scal(g_\tau) = \frac{n \Ric(\omega_\tau) \wedge \omega_\tau^{n-1}}{\omega_\tau^n}
$$
of the reduced metric $g_\tau$ can be computed by means of data on $M$.
Indeed the following holds: 
\begin{cor}\label{cor::scalred}
The function
$$ R = \scal(g) + 2 \Delta \log |V| + \frac{2}{|V|^2} \left( \Delta \mu - JV \log |V| \right)^2 + \frac{JV}{|V|^2} \left( \Delta \mu - JV \log |V| \right) $$
satisfies $R_\tau = \scal(g_\tau)$.
\end{cor}
\begin{proof}
Here the main point is that the trace of $\rho$ with respect to $\omega$ descends to the trace of $\rho_\tau$ with respect to $\omega_\tau$. The thesis then follows directly from Proposition \ref{prop::riccired}.

First of all notice that by Cartan formula, circle invariance and by the identity $JV(\mu)=|V|^2$ it follows
\begin{equation}
i_V \rho = i_V \Ric(\omega) - di_Vd^c \log |V| - d i_V \left( \frac{\Delta \mu - JV \log |V|}{ |V|^2} d^c \mu \right) = i_V \Ric(\omega) + d \Delta \mu = 0,
\end{equation}
where the last equality is quite standard, but we include the proof for convenience of the reader.
\begin{eqnarray}
\nonumber i_V \Ric(\omega) &=& i_V \left( \Ric(\omega) - \pi^* \Ric(\sigma) \right) \\
\nonumber &=& - \frac{1}{2} i_V d d^c \log \frac{\omega^{n+1}}{\pi^* \sigma^n \wedge d \log s \wedge d^c \log s} \\
\nonumber &=& - \frac{1}{2} d JV \log \frac{\omega^{n+1}}{\pi^* \sigma^n \wedge d \log s \wedge d^c \log s} 
\\
\nonumber &=& - \frac{1}{2} d \frac{L_{JV}(\omega^{n+1})}{\omega^{n+1}} 
\\
\nonumber &=& - \frac{1}{2} d \frac{(n+1) dd^c \mu \wedge \omega^n}{\omega^{n+1}} \\
\label{eq::deltamumoment}&=& - d \Delta \mu.
\end{eqnarray}

Arguing as in the proof of Lemma \ref{lem::MAred}, one has
$$ \rho \wedge \omega^n = \eta \wedge d \log s \wedge d^c \log s,$$
where $4\eta = i_{JV} i_V (\rho \wedge \omega^n)$ is a section of $\pi^* K_M$.
Thus by easy calculations one finds
\begin{equation}\label{rhowomega}
\rho \wedge \omega^n
= \frac{n}{4} \left( i_{JV}\rho \wedge d\mu \wedge \omega^{n-1} + |V|^2 \rho \wedge \left(\omega - \frac{d\mu \wedge d^c \mu}{|V|^2} \right)^{n-1} \right)\wedge d \log s \wedge d^c \log s.
\end{equation}
In the same way one gets
\begin{equation}\label{omegan+1}
\omega^{n+1} = \frac{n+1}{4} |V|^2 \left( \omega - \frac{d\mu \wedge d^c \mu}{|V|^2} \right)^n \wedge d \log s \wedge d^c \log s. 
\end{equation}
Combinig \eqref{rhowomega} and \eqref{omegan+1} then gives
\begin{equation}
\frac{(n+1) \rho \wedge \omega^n}{\omega^{n+1}}
= \frac{n|V|^{-2}i_{JV}\rho \wedge d\mu \wedge \omega^{n-1} + n\rho \wedge \left(\omega - \frac{d\mu \wedge d^c \mu}{|V|^2} \right)^{n-1}}{\left( \omega - \frac{d\mu \wedge d^c \mu}{|V|^2} \right)^n}
\end{equation}
and reducing at $\tau$, recalling that $\rho_\tau = \Ric (\omega_\tau)$, by Proposition \ref{prop::riccired} one then has
\begin{equation}
\left( \frac{(n+1) \rho \wedge \omega^n}{\omega^{n+1}} \right)_\tau
= \frac{n \rho_\tau \wedge \omega_\tau^{n-1}}{\omega_\tau^n}
= \scal(g_\tau).
\end{equation}
The thesis follows from equation above after observing that the trace of $\rho$ with respect to $\omega$ coincides with the function $R$ on the statement. Indeed one has
\begin{eqnarray*}
\frac{(n+1) \rho \wedge \omega^n}{\omega^{n+1}}
&=& \frac{(n+1) \left(\Ric(\omega) + dd^c \log |V| + d\left( \frac{\Delta \mu - JV \log |V|}{ |V|^2} d^c \mu \right)\right)\wedge \omega^n}{\omega^{n+1}} \\
&=& \scal(g) + 2 \Delta \log |V| + \frac{2}{|V|^2}(\Delta \mu - JV \log |V|)\Delta \mu + JV \frac{\Delta \mu - JV \log |V|}{ |V|^2} \\
&=& R.
\end{eqnarray*}
\end{proof}

\section{Geometric flows}

In this section we consider the case when the reduced K\"ahler forms induce some geometrically meaningful path of cohomologous metrics on a compact manifold $M$.
Thanks to Duistermaat-Heckman \cite{DuiHec}, this forces to assume $P$ to be globally a product:
$$ P = M \times \mathbf C^*.$$

By Lemma \ref{lem::HamAct}, choosing a K\"ahler metric $\omega$ on $P$ making Hamiltonian the standard circle action with moment map $\mu$ is the same as choosing a triple $(\sigma,\phi,c)$ constituted by a K\"ahler form $\sigma$ on $M$, a smooth invariant function $\phi$ on $P$, and a real constant $c$ satisfying suitable compatibility conditions with $\omega$ and $\mu$.
Moreover, by Lemma \ref{prop::redform} the reduced metrics are given by 
$$ \omega_\tau = \sigma + dd^c \psi_\tau, $$
where $\psi_\tau$ is the reduction of the invariant function $\phi + \frac{\mu-c}{2} \log s$ on $P$. 
In particular $\psi_\tau$ defines a path in the space of K\"ahler potential
$$ \mathcal H_\sigma = \left\{ \varphi \in C^\infty(M) \, | \, \sigma_\varphi=\sigma + dd^c \varphi >0 \right\}. $$
Note that a fixed path $\omega_\tau$ of K\"ahler metrics can induce many different paths in $\mathcal H_\sigma$. Indeed for any function $h=h(\tau)$, one has 
$$ \omega_\tau = \sigma + dd^c (\psi_\tau + h(\tau)).$$

On the other hand, given the path of reduced metrics $\{\omega_\tau\}$ one can ask if a path $\{ \tilde \omega_t\}$ obtained just by reparametrization $\tau=\tau(t)$ of the time comes as path of reduced metrics on $P$. This is the case, and it holds the following

\begin{prop}\label{prop::reptau}
Let $f : \mu(P) \to \mathbf R$ be any smooth function invertible on its image. The path of metrics $\{\omega_{f^{-1}(t)}\}$ on $M$ can be realized as path induced by reduction of the metric $\tilde g$ associated to the K\"ahler form
$$ \tilde \omega = \omega + dd^c \Psi, $$
where $\Psi$ is a smooth invariant function on $P$ satisfying
$$ JV(\Psi) = f(\mu)-\mu. $$
\end{prop}
\begin{proof}
Since $i_V \tilde \omega = d (\mu + JV(\Psi))$, a moment map of the metrics $\tilde g$ is given by $$ \tilde \mu = \mu + JV(\Psi) = f(\mu), $$
whence the thesis follows.
\end{proof}

\subsection{Geodesics}

As shown by Mabuchi \cite{Mab}, letting
$$ \|v\|^2 = \int_M v^2 \frac{\sigma_\varphi^n}{n!}, \qquad v \in T_\varphi\mathcal H_\sigma $$
defines a Riemannian metric on the space of K\"ahler potentials $\mathcal H_\sigma$.
It is not difficult to see that a path $\left\{\varphi_t\right\} \subset \mathcal H_\sigma$ is a geodesic if it satisfies the equation
\begin{equation}\label{eq::geodesiceq}
\varphi_t'' - \frac{1}{2}|\nabla \varphi_t'|^2 = 0,
\end{equation}
where the prime denotes the derivative with respect to the time parameter $t$, and gradient and norm are taken with respect the variable metric $\sigma_{\varphi_t}$.

\begin{thm}\label{thm::Geodesics}
Suppose that for some smooth function $h = h(\mu)$ on $P$, the metric $g$ satisfies
\begin{equation}\label{eq::Geodesics}
\Delta s = h s.
\end{equation}
Then the family of reduced metrics $g_\tau$ induces a geodesic in $\mathcal H_\sigma$.
\end{thm}
\begin{proof}
By Proposition \ref{prop::redform} the reduced K\"ahler form is given by
$$ \omega_\tau = \sigma + dd^c \psi_\tau, $$
where $ \psi = \phi + \frac{\mu-c}{2} \log s $.
Given a smooth function $\kappa(\tau)$, the path $\psi_\tau - \kappa(\tau)$ defines a geodesic in $\mathcal H_\sigma$ if and only if 
\begin{equation}
\psi_\tau'' - \frac{1}{2} \left| \nabla_\tau \psi_\tau' \right|_\tau^2 = \kappa''(\tau),
\end{equation}
where the prime denotes the derivative with respect to $\tau$.
Using the identity $\Delta e^\varphi = e^\varphi \left( \Delta \varphi + \frac{1}{2} |\nabla \varphi|^2 \right)$, one readily verifies that the equation above is equivalent to:
\begin{equation}\label{eq::equivgeod}
\psi_\tau'' - e^{-\psi_\tau'} \Delta_\tau e^{\psi_\tau'} + \Delta_\tau \psi_\tau' = \kappa''(\tau).
\end{equation}
Thanks to Lemma \ref{lem::dertau} we have:
\begin{equation}\label{eq::derivphitau}
\psi_\tau' =
(\log s)_\tau, \qquad \psi_\tau'' =
\left(-\frac{2}{|V|^2}\right)_\tau,
\end{equation}
whence, by means of Corollary \ref{cor::Deltared}, one can infer that:
\begin{equation}
\Delta_\tau e^{\psi_\tau'} =
\left( \Delta s + \left( \Delta \mu - JV \log |V| \right) \frac{2s}{|V|^2} - \frac{2s}{|V|^2} \right)_\tau
\end{equation}
and
\begin{equation}\label{eq::deltaphi'}
\Delta_\tau \psi_\tau' =
\left( \Delta \log s + \left( \Delta \mu - JV \log |V| \right) \frac{2}{|V|^2} \right)_\tau.
\end{equation}
Substituting in \eqref{eq::equivgeod}, recalling that $\Delta \log s = 0$, then gives
\begin{equation}
\left( - \frac{\Delta s}{s} \right)_\tau = \kappa''(\tau),
\end{equation}
which is equation \eqref{eq::Geodesics} after choosing $h(\tau) = - \kappa''(\tau)$. 
\end{proof}

\begin{rem}
It is clear from the proof of the result above that $h$ appears in connection to the fact that reduced potentials $\psi_\tau$ may define a geodesic only up to normalization.
\end{rem}

\subsection{Calabi flow and its variants}

\begin{thm}\label{thm::Calabiflow}
Suppose that for some smooth function $h = h (\mu)$ on $P$, the metric $g$ satisfies
\begin{equation}\label{eq::Calabiflow}
R = \log s + h,
\end{equation}
where 
\begin{equation}
R = \scal(g) + 2 \Delta \log |V| + \frac{2}{|V|^2} \left( \Delta \mu - JV \log |V| \right)^2 + \frac{JV}{|V|^2} \left( \Delta \mu - JV \log |V| \right).  
\end{equation}
Then the family of reduced metrics $g_\tau$ is a solution of the Calabi flow on $M$:
\begin{equation}\label{eq::Calabiflowforms}
\frac{\partial \omega_\tau}{\partial \tau} = \frac{1}{2} dd^c \scal(g_\tau).
\end{equation}
\end{thm}
\begin{proof}
By Proposition \ref{prop::redform} the reduced K\"ahler form is given by
$$ \omega_\tau = \sigma + dd^c \psi_\tau, $$
where $ \psi = \phi + \frac{\mu-c}{2} \log s $.
Equation \eqref{eq::Calabiflowforms} is then clearly equivalent to 
\begin{equation}\label{Calfloweqeq}
dd^c \left( \frac{\partial \psi_\tau}{\partial \tau} - \frac{1}{2} \scal(g_\tau) \right) = 0. 
\end{equation}
By Corollary \ref{cor::scalred} one has $\scal(g_\tau) = R_\tau$.
On the other hand, by the identity $2JV(\psi)=\log s$ and Lemma \ref{lem::dertau} it follows $\frac{\partial \psi_\tau}{\partial \tau} = \frac{1}{2}\left(\log s\right)_\tau$.
Thus \eqref{Calfloweqeq} turns out to be equivalent to
\begin{equation}\label{eq::CalabiFlowEquiv}
dd^c \left( \log s - R \right)_\tau = 0,
\end{equation}
and this equation is satisfied whenever \eqref{eq::Calabiflow} holds, since $h_\tau = h(\tau)$ is constant on $M$.
\end{proof}

\begin{rem}
The appearance of the function $h$ in the result above is a consequence of the fact that the reduced potentials $\psi_\tau$ are defined up to a constant on $M$ depending on $\tau$.
On the other hand, given a K\"ahler form $\omega$ on $P$, there is a unique function $h$ candidate to solve equation \eqref{eq::Calabiflow}. Indeed, after reducing at $\tau$ that equation, integrating over $M$ and dividing by the volume of $\omega_\tau$ (which is independent of $\tau$) gives
\begin{equation}
h(\tau) = \lambda - \frac{\int_M \log s_\tau \omega_\tau^n}{\int_M \omega_\tau^n},
\end{equation}
where $\lambda = \frac{n \, c_1(M) \cup [\sigma]^{n-1}}{[\sigma]^n}$ is the mean scalar curvature of the K\"ahler class $[\sigma]$.
\end{rem}

\begin{rem}
Equation \eqref{eq::CalabiFlowEquiv} suggests there exist K\"ahler metrics on $P$ for which condition \eqref{eq::Calabiflow} may fail, but still the reduced metrics induce the Calabi flow on the reduction. Clearly in this case the reduced manifold is non-compact.  
\end{rem}

The pseudo-Calabi flow has been introduced and studied by Chen and Zheng \cite{CheZhe}.
\begin{thm}\label{thm::pseudoCalabiflow}
Suppose that $g$ satisfies the following equation
\begin{equation}\label{eq::pseudoCalabiflow}
R + \frac{2}{|V|^2} \left( \Delta \mu - JV \log |V| \right) = \lambda,
\end{equation}
where 
\begin{equation}
R = \scal(g) + 2 \Delta \log |V| + \frac{2}{|V|^2} \left( \Delta \mu - JV \log |V| \right)^2 + \frac{JV}{|V|^2} \left( \Delta \mu - JV \log |V| \right),
\end{equation}
and $\lambda = \frac{n \, c_1(M) \cup [\sigma]^{n-1}}{[\sigma]^n}$. 
Then the family of reduced metrics $g_\tau$ is a solution of the pseudo-Calabi flow on $M$:
\begin{equation}\label{eq::pseudoCf}
\left\{
\begin{array}{l}
\omega_\tau = \sigma + dd^c \psi_\tau \\
\Delta_\tau \psi_\tau' + \scal(g_\tau) = \lambda.
\end{array}
\right.
\end{equation}
where the prime denotes the partial derivative with respect to $\tau$. 
\end{thm}
\begin{proof}
By Proposition \ref{prop::redform} the reduced K\"ahler form is given by
$$ \omega_\tau = \sigma + dd^c \psi_\tau, $$
where $ \psi = \phi + \frac{\mu-c}{2} \log s $.
Thanks to \eqref{eq::deltaphi'}, the path $\psi_\tau$ is a solution of \eqref{eq::pseudoCf} if and only if 
\begin{equation}
\left( \Delta \log s +  \frac{2}{|V|^2}\left( \Delta \mu - JV \log |V| \right) \right)_\tau + \scal(g_\tau) = \lambda,
\end{equation}
thus the thesis follows by Corollary \ref{cor::scalred} and observing that $\Delta \log s = 0$.
\end{proof}

\begin{rem}
It is a notable fact that for the pseudo-Calabi flow, one need not introduce the function $h$ in order to renormalize the reduced potentials $\psi_\tau$.
\end{rem}

\subsection{K\"ahler-Ricci flow}

The approach to K\"ahler-Ricci flow by means of symplectic reduction is due to La Nave and Tian with the introduction of the V-soliton equation \cite{LT}. Here we recover the equation of the un-normalized K\"ahler-Ricci flow by results of section \ref{sec::Hamcircact}.

\begin{thm}\label{thm::KRflow}
Assume the metric $g$ satisfies
\begin{equation}\label{eq::KRflow}
\Ric(\omega) + dd^c \log |V| + d \left( \frac{\Delta \mu - JV \log |V| + 1}{|V|^2} d^c \mu \right) = 0.
\end{equation}
Then the family of reduced metrics $g_\tau$ is a solution of the K\"ahler-Ricci flow on $M$:
\begin{equation}\label{eq::KRF}
\frac{\partial \omega_\tau}{\partial \tau} = - \Ric(\omega_\tau).
\end{equation}
\end{thm}
\begin{proof}
By Proposition \ref{prop::redform} the reduced K\"ahler form is given by
$$ \omega_\tau = \sigma + dd^c \psi_\tau, $$
where $ \psi = \phi + \frac{\mu-c}{2} \log s $.
Equation \eqref{eq::KRF} is clearly equivalent to 
\begin{equation}\label{KRFeq}
dd^c \left( \frac{\partial \psi_\tau}{\partial \tau}\right) + \Ric(\omega_\tau)  = 0. 
\end{equation}
On the other hand, by the identity $2JV(\psi)=|V|^2 \log s$ and Lemma \ref{lem::dertau} it follows $\frac{\partial \psi_\tau}{\partial \tau} = \frac{1}{2}\left(\log s\right)_\tau$.
Thus \eqref{eq::KRF} turns out to be equivalent to
\begin{equation}\label{eq::KRFEquiv}
\frac{1}{2}dd^c \log s_\tau + \Ric(\omega_\tau) = 0.
\end{equation}
By Lemma \ref{lem::dcFtau} and the equality $JV \log s = -2$ it follows
$$ \frac{1}{2}dd^c \log s_\tau = d \left(d^c \log s + \frac{d^c \mu}{|V|^2} \right)_\tau, $$
on the other hand, by Proposition \ref{prop::riccired} one has $\Ric(\omega_\tau) = \rho_\tau$, where
$$ \rho = \Ric(\omega) + dd^c \log |V| + d\left( \frac{\Delta \mu - JV \log |V|}{ |V|^2} d^c \mu \right). $$
Thus the thesis follows by substituting in \eqref{eq::KRFEquiv}.
\end{proof}

The following result is due to La Nave-Tian \cite[Theorem 3.7]{LT}. Here we give a proof resting on results of section \ref{sec::Hamcircact}. The reader is referred to their work for more details.

\begin{thm}\label{thm::nKRflow}
Suppose that for some smooth function $f = f (\mu)$ on $P$, the metric $g$ satisfies
\begin{equation}\label{eq::nKRflow}
\Ric(\omega) + dd^c \left( \log |V| + f \right) = \lambda \omega,
\end{equation}
being $\lambda = \frac{n \, c_1(M) \cup [\sigma]^{n-1}}{[\sigma]^n}$.
Then the family of reduced metrics $g_\tau$ is a solution of the K\"ahler-Ricci flow on $M$:
\begin{equation}\label{eq::nKRF}
\frac{\partial \omega_\tau}{\partial t} = - \Ric(\omega_\tau) + \lambda \omega_\tau.
\end{equation}
where $ \tau(t) = \frac{a+be^{\lambda t}}{\lambda} $ for some constants $a,b \in \mathbf R$.
\end{thm}
\begin{proof}
Recall that $i_V \Ric(\omega) + d \Delta \mu = 0$ as proved in \eqref{eq::deltamumoment}.
Thus contracting \eqref{eq::nKRflow} with $V$ gives
\begin{equation}
- d\Delta \mu + d \left( JV \log |V| + JV(f) \right) = \lambda d\mu,  
\end{equation}
whence, by connectedness of $P$, it follows that there exists a real constant $a$ such that
\begin{equation}\label{Deltamu-JVlogV}
\Delta \mu - JV \log |V| = JV(f) - \lambda \mu + a.
\end{equation}
On the other hand, thanks to the expression of $\rho$ in Proposition \ref{prop::riccired}, equation \eqref{eq::nKRflow} turns out to be equivalent to
$$
\rho + d \left(d^c f - \frac{\Delta \mu - JV \log |V|}{|V|^2} d^c\mu \right) = \lambda \omega,
$$
and by \eqref{Deltamu-JVlogV} even to
\begin{equation}\label{eq::nKRfloweq}
\rho + d \left(d^c f - \frac{JV(f)-\lambda \mu + a}{|V|^2} d^c\mu \right) = \lambda \omega.
\end{equation}

Arguing as at the beginning of the proof of Theorem \ref{thm::KRflow} one finds
$$\frac{\partial \omega_\tau}{\partial \tau} = d \left(d^c \log s + \frac{d^c \mu}{|V|^2} \right)_\tau,$$
whence it follows that $d \left(\frac{\lambda\mu-a}{|V|^2}d^c\mu \right)$ reduces to $(\lambda \tau-a) \frac{\partial \omega_\tau}{\partial \tau}$.
On the other hand, by Lemma \ref{lem::dertau} the form $d\left(d^cf - \frac{JV(f)}{|V|^2}d^c\mu\right)$ reduces to $dd^c f(\tau) = 0$.
Thus reducing \eqref{eq::nKRfloweq} gives
$$ \Ric(\omega_\tau) + (\lambda \tau - a) \frac{\partial \omega_\tau}{\partial \tau} = \lambda \omega_\tau,$$
and the statement follows readily by reparametrizing $\tau$ with $t$.
\end{proof}
\section{The structure of the total space and the converse theorems}\label{section::total-space}

The following result says that any path of cohomologous K\"ahler metrics can be realized as a family of reduced metrics from a bigger K\"ahler manifold.

\begin{thm}
Let $g_t$ be a smooth path of cohomologous K\"ahler metrics on a compact complex manifold $M$ with $t\in [0,T)$, for some $T \in (0,+\infty]$.
Then there is $r\in[0,1)$ and a circle invariant K\"ahler metric on $P=M\times A$, where $A = \{ r<|w|<1\} \subset \mathbf C$, inducing the path $g_t$ on $M$ via K\"ahler reduction.
\end{thm}

\begin{proof}
Let $\sigma_t$ be the K\"ahler form of $g_t$.
By $dd^c$-lemma there exists a smooth path of $\sigma$-plurisubharmonic functions $\psi_t$ such that $\sigma_t = \sigma + dd^c \psi_t$, being $\sigma=\sigma_0$.
Note that $\psi_t$ is just defined up to an additive constant possibly depending on $t$. We will make use of this arbitrariness later.

Recalling that $\pi : P \to M$ is the projection on the first factor, and $s$ is the smooth function on $P$ defined by $s(p)=|w|^2$ if $p=(x,w)$, let $F$ be the function on $[0,T) \times P$ defined by 
\begin{equation}
F(t,p) = \psi_t(\pi(p)) - \frac{t}{2} \log s(p).
\end{equation}

We look for a circle invariant function $\phi$ on $P$ such that 
\begin{equation}\label{eq::Feq}
F(JV(\phi),p) = \phi.
\end{equation}
for all $p \in P$, and $JV(\phi)(p)=0$ if $s(p)=1$.
This equation is clearly equivalent to 
\begin{equation}\label{eq::F=phiequiv}
\phi + \frac{JV(\phi)}{2} \log s = \psi_{JV(\phi)}\circ \pi,
\end{equation}
thus by Proposition \ref{prop::redform} and Lemma \ref{lem::HamAct} the function $\psi_\tau$ is the K\"ahler potential of reduced metrics whenever one can find a circle invariant solution $\phi$ to \eqref{eq::Feq} satisfying $\pi^* \sigma + dd^c \phi >0$. 
If a solution to \eqref{eq::Feq} exists, letting $\mu=JV(\phi)$, by \eqref{eq::ddcphi} and \eqref{eq::F=phiequiv} it holds
\begin{eqnarray}
dd^c \phi
&=& d\left( \left.\pi^* d^c\psi_t + \frac{\partial \psi_t \circ \pi}{\partial t} d^c\mu \right|_{t=\mu} - \frac{1}{2} \log s \, d^c\mu - \frac{\mu}{2}d^c\log s \right) \nonumber \\
&=& d\left( \left.\pi^* d^c\psi_t \right|_{t=\mu} - \frac{\mu}{2}d^c\log s \right) \nonumber \\
&=& \left.\pi^* dd^c\psi_t \right|_{t=\mu} + d\mu \wedge \left( \left. \pi^* d^c \frac{\partial \psi_t}{\partial t} \right|_{t=\mu} -\frac{1}{2} d^c \log s \right) \nonumber \\
&=& \left.\pi^* dd^c\psi_t - \frac{\partial^2 \psi_t \circ \pi}{\partial t^2} d\mu \wedge d^c\mu \right|_{t=\mu}, \label{eq::ddcphipsi}
\end{eqnarray}
where the second and last equalities follow from the identity $\frac{1}{2} \log s = \left. \frac{\partial \psi_t \circ \pi}{\partial t} \right|_{t=\mu}$, which in turn descends readily from applying $JV$ to \eqref{eq::F=phiequiv}.
Thus for any vector field $X$ on $P$ one has
$$ (\pi^* \sigma + dd^c \phi) (X,JX) = \left.(\sigma+dd^c \psi_t)\right|_{t=\mu} (\pi_*X , \pi_*JX) - \left. \frac{\partial^2 \psi_t \circ \pi}{\partial t^2} \right|_{t=\mu} (X(\mu)^2 + JX(\mu)^2), $$
whence it is clear that $\pi^* \sigma + dd^c \phi>0$ if and only if $\frac{\partial^2 \psi_t \circ \pi}{\partial t^2} <0$.
The latter inequality can be made to hold with no additional assumption simply replacing $\psi_t$ by $\psi_t + a_t$ where
$$ a_t = -t^2 - \int_0^t \int_0^v \sup_M \frac{\partial^2 \psi_u}{\partial u^2} dudv.$$

It remains then to show that \eqref{eq::Feq} can be solved.
This follows by a standard application of method of characteristics.
It is not difficult to see that projected characteristics are integral curves of $JV$ in $P$.
\end{proof}
\begin{rem}
The Theorem above clearly implies that any of geometric flows considered in the previous section-- with the notable exception of the K\'ahler-Ricci flow for which $\frac{\partial h}{\partial \tau}$ is a nontrivial class (see formula \eqref{eq::connection} below), i.e. when the initial metric is not canonical (e.g., the flow on projective manifolds with non ample canonical bundles), which is treated in \cite{LT}-- can be realized as a path of reduced K\"ahler metric coming from a suitable K\"ahler metric on $P$.
\end{rem}
This can be proved in a more geometrical way, as illustrated for the Calabi flow in the following subsection. This proof is less elegant and less general (in that one needs to reformulate the proof in avery case) but it has the advantage of working even when one is forced to take an $M$ consisting of a family of non-trivial principal $S^1$-bundles.

\subsection{Converse to Theorem \ref{thm::Calabiflow}}

Let be given a Hamiltonian holomorphic circle action on a K\"ahler manifold $(P,\omega, J)$ with associated metric $g$ and with moment map $\mu$. Fix a regular value $a$ of the moment map.
Let $z_1,\cdots, z_m$, with $m=n-1$, be holomorphic coordinates on the quotient
manifold $M=\mu^{-1}(a)/S^1$, and let $\tau $ be the "moment map coordinate", i.e.,
$\mu=\tau$. By definition, we have $d\tau  =i_{V} \omega$.

Let $Q$ be the {\it horizontal distribution}, i.e., the hortogonal complement (with respect to the metric $g$) of the span of $V$ and $JV$ (clearly the quotient map $\pi_a: \mu ^{-1}(a) \to M$ induces an isomorphism $d\pi _a: Q\to TM$). Then we can define a ~1-form $\theta$ by:
$$\theta (V)=1, ~\theta (J V)=0, ~\theta \big |_{Q}=0,$$
One can show (cf. Lemma 3.1 in \cite{LT}) that:
\begin{lem}
For the above local coordinates, we have $g(dz_i, d\tau)=0$, $g(dz_i,\theta)=0$ and $g(\theta ,d\tau)=0$,
where $g$ also denotes the induced metric on the cotangent bundle of $M$.
In particular, in these coordinates $g$ and $\omega$ take the form (reap.):
\begin{equation}\label{invmetric}
g = h_{i\bar j} dz_i d\bar z _j+ w \, d \tau ^2 +
\frac{1}{w}\, \theta^2\end{equation}
and 
\begin{equation}\label{invform}
\omega =\sqrt{-1} h_{i\bar j} dz_i \wedge d\bar z _j-d\tau \wedge \theta \end{equation}
\end{lem}
where $\frac{1}{w}= |V|^2$.

One can see that $J$ preserves $Q^{1,0}:= Q\cap T{(1,0)}P$ and that $\alpha:=w\, d\tau -
\sqrt{-1} \theta $ is of type $(1,0),$. We can then rewrite $g$ as
$$\omega =\sqrt{-1}\, h_{i\bar j} dz_i d\bar z _j+|V|^2 \,\alpha \wedge \bar \alpha.$$
Also, we have the decomposition: $T^{(1,0)}M = Q^{(1,0)} \oplus \langle \alpha \rangle$.
Following this decomposition one can write, for any $S^1$-invariant function $f$:
$$\partial f = \partial ^h f+ JV(f) \, \alpha \qquad \qquad \bar \partial f =\bar  \partial ^h f+ JV(f) \, \bar \alpha$$
and these identities define $\partial ^h f\in Q^{(1,0)}$ and $\bar \partial ^h f\in Q^{(0,1).}$
\noindent
We have the following import an formula about the connection ~1-form of the principle $S^1$-bundle:
\begin{lem}\label{connection}
One has:
\begin{equation}\label{eq::connection}
d \theta = \sqrt{-1}\left\{ - \frac{\partial {h_{i\bar
    j}}}{\partial \tau } dz_i\wedge d \bar z _j +\partial ^h \log \left(\frac{1}{|V|^2}\right) \wedge\bar  \alpha +\alpha \wedge \bar \partial ^h \log \left(\frac{1}{|V|^2}\right) -d  \log \left(\frac{1}{|V|^2}\right)\wedge \theta \right\}\end{equation}
\end{lem}
\begin{proof}
This is Lemma 3.2 in \cite{LT}, after using that:
$$\partial ^h \log \left(\frac{1}{|V|^2}\right) \wedge\bar  \alpha= \partial ^h w\wedge d\tau +\partial ^h \log w\wedge \theta$$
and the analogous:
$$\alpha \wedge \bar \partial ^h \log \left(\frac{1}{|V|^2}\right)= d\tau \wedge \bar  \partial ^h w+ \bar \partial ^h \log w\wedge \theta.$$
\end{proof}

We can prove the sought after converse to Theorem \ref{thm::Calabiflow}:

\begin{thm}\label{theorem::total-space}
Maintaining notations as in Theorem \ref{thm::Calabiflow}, let $(M, \omega(t))$ be a solution of the Calabi-flow--i.e., eq. \eqref{eq::Calabiflowforms}-- and assume that along the flow:
$$\frac{\partial  \scal(g_\tau)}{\partial \tau}>-c$$
for some uniform $c>0.$
Then there exists a metric $g$ and a function $h=h(\mu)$ such that:
$$R= \log s+h$$
\end{thm}
\begin{proof}
Let $\omega _\tau =  \sqrt{-1} \, h_{i\bar j}\, dz_i\wedge d \bar z _j$.
Naturally we set $\tau=- C t+C_0$ where $C$ is to be determined ($C_0$ is arbitrary).
Imposing equation \eqref{eq::connection}, simply define (for now just formally as we don't know as of yet that thus defined, $d\theta$ is closed):
$$d \theta = \sqrt{-1}\left\{ - \frac{\partial {(g_\tau)_{i\bar
    j}}}{\partial \tau } dz_i\wedge d \bar z _j +\partial ^h \log \left(\frac{1}{|V|^2}\right) \wedge\bar  \alpha +\alpha \wedge \bar \partial ^h \log \left(\frac{1}{|V|^2}\right) -d  \log \left(\frac{1}{|V|^2}\right)\wedge \theta \right\}$$
where $\frac{1}{|V|_g^2}$ is yet to be defined.
 
 Then clearly:
 $$d\theta\big | _{\tau=const}=\frac{\partial \omega _\tau}{\partial \tau}$$
 which is in the trivial class, due to the fact that $\omega(t)$ satisfies the Calabi flow. We can therefore take $P=M\times A$ where $A\subset \mathbb C^*$ is an annulus and the rest is clear, modulo defining the value of $|V|_g^2$.
Equations \eqref{Calfloweqeq} and \eqref{eq::CalabiFlowEquiv} imply that the equation $R=\log s+h$ is satisfied in the {\it horizontal directions}. It is then readily seen that it must be,  in order for equation $R=\log s+h$ to be satisfied, one must have:
$$\nabla \left( JV\left( \scal(g_\tau) -h\right) \right)=0$$
and also that in fact one must take the constant $C$ for which $\tau= -C t+ C_0$ to be:
$$-C:= JV\left( \scal(g_\tau) -h\right) .$$
This very equation defines $|V|_g^2$ though, as one has $JV= |V|_g^2\, \frac{\partial }{\partial \tau}$ which yields:
\begin{equation} \label{eq::definition-of-w} |V|_g^2= \frac{-C}{\frac{\partial }{\partial \tau}\left( \scal (g_\tau)-h\right)}= \frac{1}{\frac{\partial }{\partial t}\left( \scal (g_\tau)-h\right)}\end{equation}
One can therefore choose $h$ such that $|V|_g^2>0$, provided:
$$\frac{\partial  \scal(g_\tau)}{\partial \tau}>-c$$
for some uniform $c>0.$
It is now also easy to see from the Calabi flow equation that the form $d\theta$ is closed --since $\frac{1}{|V|_g^2}= \frac{\partial }{\partial t}\left( \scal (g_\tau)-h\right)$. In fact, the closedness of $d\theta$ is equivalent to requiring:
$$\frac{\partial^2 {(g_\tau)_{i\bar
    j}}}{\partial \tau ^2} =- \frac{\partial  ^2}{\partial z_i \partial \bar z_j}\left( \frac{1}{|V|_g^2}  \right)$$
 \noindent
 Using equation \eqref{eq::definition-of-w} and the equation of the Calabi flow, this is readily seen to be equivalent to the condition:
 $$\frac{\partial }{\partial \tau }\left( \frac{\partial  ^2( \scal (g_\tau)-h)}{\partial z_i \partial \bar z_j}\right)= - \frac{\partial  ^2}{\partial z_i \partial \bar z_j}\left(  \frac{\frac{\partial }{\partial \tau}\left( \scal (g_\tau)-h\right)} {-C} \right)$$
 which clearly holds.
\end{proof}

The converse to Theorem \ref{thm::nKRflow} has been treated in \cite{LT}, while the converse to Theorem \ref{thm::KRflow} is immediately adapted form there.

\end{document}